\documentclass[11pt,reqno,a4paper]{amsart}
\usepackage{amsfonts,amsmath,times,amsthm,amssymb}
\usepackage{tikz}
\usepackage{cool} %Package for generalized hypergeometric functions

\newtheorem{theorem}{Theorem}
\newtheorem{lemma}{Lemma}

\newtheorem{prop}{Proposition}
\newtheorem{corollary}{Corollary}

\renewcommand{\E}{\mathbb{E}}
\newcommand{\Var}{\mathbb{V}{\rm ar}}
\newcommand{\Prob}{\mathbb{P}}
\newcommand{\Polya}{P\'{o}lya}
\newcommand{\field}{\mathbb{F}}
\newcommand{\indicator}{\mathbb{I}}
\newcommand{\given}{\, \vert \,}
\newcommand{\convL}{\, \overset{L_1}{\longrightarrow} \,}
\newcommand{\convP}{\, \overset{P}{\longrightarrow} \,}

\newcommand{\convD}{\, \overset{D}{\longrightarrow} \,}

\newcommand{\Expo}{{\rm Exp}}

\begin{document}
\begin{center}
	{\Large \bf  
		On several properties of plane-oriented recursive trees}
	
	\bigskip
	{\bf Panpan Zhang}
	
	\bigskip
	{\tiny Department of Biostatistics, Epidemiology and Informatics, Perelman School of Medicine, University of Pennsylvania, Philadelphia, PA 19104, U.S.A.}
	
	\bigskip
	
	\today
\end{center}

\bigskip\noindent
{\bf Abstract.}
In this paper, several properties of plain-oriented recursive trees (PORTs) are uncovered. Specifically, we investigate the degree profile of a PORT by determining the exact probability mass function of the degree of a node with a fixed label. We compute the expectation and variance of the degree variable via a \Polya\ urn approach. In addition, we look into a topological index, the Zagreb index, of this class of trees. We calculate the exact first two moments of the Zagreb index by using recurrence methods. We also provide several evidence in favor of our conjecture that the Zagreb index of PORTs do not follow a Gaussian law asymptotically. Lastly, we determine the limiting degree distribution in Poissonized PORTs, and show that it is exponential after being properly scaled.

  

\section{Introduction}
In graph theory, a {\em tree} refers to a connected structure with no cycles~\cite[page 24]{Berge}. A {\em random recursive tree} is an unordered labeled tree with label set such that there exists an increasing unique path from the {\em root} (the most primitive node labeled with $1$) to the node labeled with $j$ for all $2 \le j \le n$. This class of uniform recursive trees was proposed in the late 1960s, and has found applications in a plethora of areas, such as spread of epidemics~\cite{Moon}, genealogy~\cite{Najock}, and the pyramid scheme~\cite{Gastwirth}. 

In this paper, we consider a class of nonuniform random recursive trees---plane-oriented recursive trees (PORTs). A {\em plane-ordered recursive tree} is a tree in which descendants of each node are ordered. At time $n \ge 1$, we denote the structure of a PORT as $T_n$, i.e., a PORT consisting of $n$ nodes. The tree $T_n$ is obtained by starting with a single node labeled with $1$ (i.e., root). Upon each insertion point $n \ge 2$, a node labeled with $n$ joins into the tree; the probability of the newcomer (the node labeled with $n$) adjoint to the node labeled with $i$, for $1 \le i \le n - 1$, in $T_{n - 1}$  is proportional to the {\em degree}\footnote{The degree of a node is the number of edges incident with the node.} of the recruiter (the node labeled with $i$). The key feature of this class of trees is that a {\em parent} node with higher degrees is more attractive to the newcomers, which coincides with a manifestation of the economic principles---``the rich get richer'' and ``success breeds success.'' 

Precursory research on PORTs traced back to the late 1980s and the early 1990s. The exact and asymptotic moments of two degree profile random variables, the number of nodes of a given degree and the degree of a fixed node, were investigated in~\cite{Szymanski}. The distribution of the depth of nodes was determined by~\cite{Mahmoud1991}. The exact and asymptotic distribution of leaves (terminal vertices) in PORTs and subtrees (branches) were studied by~\cite{Mahmoud1993}. The asymptotic average of internal path length was characterized by~\cite{Chen}. More recently, PORTs again caught researchers' attention since its evolutionary characteristic coincides with a network property of great interest in the statistical community---{\em preferential attachment}~\cite{Barabasi}. PORTs are a special class of preferential attachment networks, of which the network index equals one; i.e., the newcomer is connected with only one parent node at each point. The joint asymptotic distribution of the numbers of nodes of different {\em outdegrees}\footnote{The outdegree of a node is the number of edges emanating out of the node.} in PORTs was shown to be normal by~\cite{Drmota, Janson}. Several other limiting results for PORTs were presented in~\cite{Hwang}. 

The rest of this paper is organized as follows. We begin with introducing some notations and preliminaries in Section~\ref{Sec:notations}. In Section~\ref{Sec:degreedist}, we determine the exact distribution of the degree of a node with a fixed label. More specifically, we develop the probability mass function via an elementary approach---two-dimensional induction---in Section~\ref{Sec:pmf} , and calculate its moments by exploiting a \Polya\ urn model in Section~\ref{Sec:moments}. In Section~\ref{Sec:Zagreb}, we look into the Zagreb index for this class of trees. This section is split into two parts. In Section~\ref{Sec:rec}, we compute the mean and variance of the Zagreb index of PORTs via recurrence methods, while in Section~\ref{Sec:normal}, we study the asymptotic distribution of the Zagreb index and conjecture that it is not normal. In Section~\ref{Sec:Pois}, we investigate the degree profile of PORTs embedded into continuous time. We find that the asymptotic distribution of the degree variable under the Poissonization framework is exponential. Lastly, we address some concluding remarks and propose some future work in Section~\ref{Sec:concl}.

\section{Notations and preliminaries}
\label{Sec:notations}
Let $D_{n, j}$ be the degree of the node with label $j$  in $T_n$, for $1 \le j \le n$. Let $\field_n$ denote the $\sigma$-field generated by the first $n$ stages of $T_n$. Many results in this paper are given in terms of {\em gamma functions}, $\Gamma(\cdot)$; see a classic text~\cite[page 47]{Feller} for its definition and fundamental properties. For a nonnegative integer $z$, the {\em double factorial} of $z$ is $z!! = \prod_{i = 0}^{\lceil z/2 \rceil - 1} (z - 2i),$ with the interpretation of $0!! = 1$. The {\em Pochhammer symbol} for the rising factorial is defined as 
$$\langle x \rangle_k = x (x + 1) \cdots (x + k - 1)$$
for any real $x$ and nonnegative integer $k$, with the interpretation of $\langle x \rangle_0 = 1$. The {\em Kronecker delta function} of two variables $s$ and $t$, denoted by $\delta_{s, t}$, equals $1$ for $s = t$; $0$, otherwise. The little $o$ and big $O$ notations define relations between two real-valued functions $f(x)$ and $g(x)$. We have $f(x) = o\bigl(g(x)\bigr)$ equivalent to $\lim_{x \to \infty} \bigl(f(x)/g(x)\bigr)= 0$ provided that $g(x) \neq 0$; On the other hand, $f(x) = O\bigl(g(x)\bigr)$ if there exists $M > 0$ and $x_0 \in \mathbb{R}$ such that $|f(x)| \le M |g(x)|$ for all $x \ge x_0$. {\em Generalized hypergeometric functions} are defined in terms of Pochhammer symbols of rising factorials; that is,
$$\Hypergeometric{p}{q}{a}{b}{z} = \sum_{s = 0}^{\infty}\frac{\langle a_1 \rangle_s \cdots \langle a_p \rangle_s}{\langle b_1 \rangle_s \cdots \langle b_q \rangle_s} \frac{z^s}{s!}.$$

Much of the study in this paper relies on an extensively-studied probabilistic model---{\em \Polya\ urn model}. We give quick words about \Polya\ urns. A two-color \Polya\ urn scheme is an urn containing balls of two different colors (say white and blue). At each point of discrete time, we draw a ball from the urn at random, observe its color and put it back in the urn, then execute some ball additions (or removals) according to predesignated rules: If the ball withdrawn is white, we add $a$ white balls and $b$ blue balls; Otherwise, the ball withdrawn is blue, in which case we add $c$ white balls and $d$ blue balls. The dynamics of the urn can thus be represented by the following {\em replacement matrix}
$$\begin{pmatrix}
a & b
\\ c & d
\end{pmatrix},$$
in which the rows from top to bottom are indexed by white and blue, and
the columns from left to right are also indexed by white and blue. We refer the interested readers to~\cite{Mahmoud2009} for a text-style elaboration of \Polya\ urns.

\section{Degree distribution}
\label{Sec:degreedist}
In this section, we investigate the degree profile of PORTs, i.e, the distribution of the degree variable $D_{n, j}$ for a fixed $1 \le j \le n$. First, we determine the distribution of $D_{n, j}$ by developing the exact expression of its probability mass function. Next, we characterize its behavior by looking into the first two moments. 
\subsection{Probability mass function}
\label{Sec:pmf}
To determine the probability mass function of $D_{n, j}$, we1 separate the cases of $\{j \ge 2\}$ and $\{j = 1\}$ for clarity. When $j = 1$, the random variable $D_{n, j} = D_{n, 1}$ refers to the degree of the root of $T_n$. The root is the originator of the tree, so it has no parent. The root is the only node in the tree that has {\em indegree}\footnote{The indegree of a node is the number of edges heading into the node} $0$.
\begin{prop}
	\label{Prop:degreej}
	For a fixed $2 \le j \le n$, we have
	\begin{equation}
	\label{Eq:degreej}
	\Prob(D_{n, j} = d) = \frac{\Gamma(d)\Gamma\left(j - \frac{1}{2}\right)\sum_{i = 0}^{d - 1} \frac{(-1)^i \Gamma\left(n - 1 - \frac{i}{2}\right)}{\Gamma(i + 1)\Gamma(d - i)\Gamma\left(j - 1 - \frac{i}{2}\right)}}{\Gamma\left(n - \frac{1}{2}\right)},
	\end{equation}
	for $d = 1, 2, \ldots, n - j + 1$.
\end{prop}

\begin{proof}
	We prove the result in this proposition by a two-dimensional induction on $n \ge j$ and $d \ge 1$. The proof progresses in the style of filling an infinite lower triangular table, in which the rows are indexed by $n$ and the columns are indexed by $d$. A (similar) graphic interpretation of the method can be found in~\cite[page 69]{Zhang}. We initialize the first column and the diagonal of the table to be the basis of the induction. The event of $\{D_{n, j} = 1\}$ for all $n \ge j$ is that the node labeled with $j$ is never chosen as a parent for any newcomer since its first appearance in the tree till time $n$. Thus, we have
	$$\Prob(D_{n, j} = 1) = \frac{2j - 2}{2j - 1} \times \frac{2j}{2j + 1} \times \cdots \times \frac{2n - 4}{2n - 3} = \frac{\Gamma(n - 1)\Gamma\left(j - \frac{1}{2}\right)}{\Gamma\left(n - \frac{1}{2}\right)\Gamma(j - 1)}.$$
	On the other hand, the event of $\{D_{n, j} = n - j + 1\}$ for all $n \ge j$ is that the node labeled with $j$ is selected as parents for newcomers at all points from $j + 1$ to $n$. It follows that
	\begin{align*}
	\Prob(D_{n, j} = n - j + 1) &= \frac{1}{2j - 1} \times \frac{2}{2j + 1} \times \cdots \times \frac{n - j}{2n - 3} 
	\\ &= \frac{\Gamma(n - j + 1) \Gamma\left(j - \frac{1}{2}\right)}{2^{n - j} \Gamma\left(n - \frac{1}{2}\right)}.
	\end{align*}
	
	Then, we assume that Equation~(\ref{Eq:degreej}) holds for all $d$ up to row $(n - 1)$ in the table. Noticing that the degree of the node labeled with $j$ increases at most by one at each point, we have
	\begingroup
	\allowdisplaybreaks
	\begin{align*}
	\Prob(D_{n, j} = d) &= \frac{d - 1}{2n - 3} \Prob(D_{n - 1, j} = d - 1) + \frac{2n - 3 - d}{2n - 3} \Prob(D_{n - 1, j} = d)
	\\ &= \frac{d - 1}{2n - 3}\frac{\Gamma(d - 1)\Gamma\left(j - \frac{1}{2}\right)\sum_{i = 0}^{d - 2} \frac{(-1)^i \Gamma\left(n - 2 - \frac{i}{2}\right)}{\Gamma(i + 1)\Gamma(d - 1 - i)\Gamma\left(j - 1 - \frac{i}{2}\right)}}{\Gamma\left(n - \frac{3}{2}\right)}
	\\ &\qquad{}+ \frac{2n - 3 - d}{2n - 3}\frac{\Gamma(d)\Gamma\left(j - \frac{1}{2}\right)\sum_{i = 0}^{d - 1} \frac{(-1)^i \Gamma\left(n - 2 - \frac{i}{2}\right)}{\Gamma(i + 1)\Gamma(d - i)\Gamma\left(j - 1 - \frac{i}{2}\right)}}{\Gamma\left(n - \frac{3}{2}\right)}
	\\ &= \frac{\Gamma(d)\Gamma\left(j - \frac{1}{2}\right)}{\Gamma\left(n - \frac{1}{2}\right)}\left[\frac{1}{2}\sum_{i = 0}^{d - 2} \frac{(-1)^i \Gamma\left(n - 2 - \frac{i}{2}\right)}{\Gamma(i + 1)\Gamma(d - 1 - i)\Gamma\left(j - 1 - \frac{i}{2}\right)}\right.
	\\ &\qquad{}+\left(n - \frac{d}{2} - \frac{3}{2}\right)\left.\sum_{i = 0}^{d - 1} \frac{(-1)^i \Gamma\left(n - 2 - \frac{i}{2}\right)}{\Gamma(i + 1)\Gamma(d - i)\Gamma\left(j - 1 - \frac{i}{2}\right)}\right]
	\\ &= \frac{\Gamma(d)\Gamma\left(j - \frac{1}{2}\right)}{\Gamma\left(n - \frac{1}{2}\right)}\left[\sum_{i = 0}^{d - 2} \left(n - 2 - \frac{i}{2}\right) \frac{(-1)^i \Gamma\left(n - 2 - \frac{i}{2}\right)}{\Gamma(i + 1)\Gamma(d - i)\Gamma\left(j - 1 - \frac{i}{2}\right)}\right.
	\\ &\qquad{}+\left(n - \frac{d}{2} - \frac{3}{2}\right)\left.\frac{(-1)^{d - 1} \Gamma\left(n - \frac{d}{2} - \frac{3}{2}\right)}{\Gamma(d)\Gamma\left(j - \frac{d}{2} - \frac{3}{2}\right)}\right].
	\end{align*}
	\endgroup
	This is equivalent to Equation~(\ref{Eq:degreej}) stated in the proposition.
\end{proof}

The probability mass function of $D_{n, j}$ is given by the sum of an alternating sequence. We split the total sum into two parts: a partial sums of odd indices and a partial sum of even indices. We then respectively evaluate the two partial sums to obtain an alternative expression of the probability mass function of $D_{n, j}$. The result is given in terms of generalized hypergeometric functions, presented in the next corollary.
\begin{corollary}
	\label{Cor:altdegree}
	For a fixed $2 \le j \le n$, we have
	\begin{align*}
	\Prob(D_{n, j} = d) &= \frac{\Gamma(d)\Gamma\left(j - \frac{1}{2}\right)}{\Gamma\left(n - \frac{1}{2}\right)} \left(\frac{\Gamma(n - 1) \Hypergeometric{3}{2}{\frac{2 - d}{2}, \frac{1 - d}{2}, 2 - j}{\frac{1}{2}, 2 - n}{1}}{\Gamma(d) \Gamma(j - 1)}\right.
	\\ &\qquad{}- \left. \frac{\Gamma(n - \frac{3}{2}) \Hypergeometric{3}{2}{\frac{3 - d}{2}, \frac{2 - d}{2}, \frac{5}{2} - j}{\frac{3}{2}, \frac{5}{2} - n}{1}}{\Gamma(d - 1) \Gamma \left(j - \frac{3}{2} \right)} \right).
	\end{align*}
\end{corollary}

The first generalized hypergeometric function (on the top row) in the result stated in Corollary~\ref{Cor:altdegree} can be further simplified for small choices of $j$. We present the probability mass functions of $D_{n, j}$ for $j = 2, 3$ as examples:
\begin{align*}
\Prob(D_{n, 2} = d) &= \frac{1}{(2n - 3) \Gamma \left(n - \frac{1}{2}\right)} \left[\sqrt{\pi}\left(n - \frac{3}{2}\right)\Gamma(n - 1)  \right.
\\ &\qquad{}- \left. (d - 1)\Gamma \left(n - \frac{1}{2}\right)\Hypergeometric{3}{2}{\frac{3 - d}{2}, \frac{2 - d}{2}, \frac{1}{2}}{\frac{3}{2}, \frac{5}{2} - n}{1} \right];
\\ \Prob(D_{n, 3} = d) &= \frac{3}{(2n - 3) \Gamma \left(n - \frac{1}{2}\right)} \left[\sqrt{\pi}\left(n - \frac{3}{2}\right) \frac{d^2 - 3d + 2n - 2}{4}\Gamma(n - 2)  \right.
\\ &\qquad{}- \left. (d - 1)\Gamma \left(n - \frac{1}{2}\right)\Hypergeometric{3}{2}{\frac{3 - d}{2}, \frac{2 - d}{2}, -\frac{1}{2}}{\frac{3}{2}, \frac{5}{2} - n}{1} \right].
\end{align*}
Simplifications for the probability mass function of $D_{n, j}$ for higher values of $j$ are also available, done in a similar manner.

Next, we look at the degree distribution of the root of a PORT. For $j = 1$, the probability mass function of $D_{n, j}$ (i.e., $D_{n, 1}$) cannot be directly derived from Equation~(\ref{Eq:degreej}). Notice that a main different difference between the root and other nodes is that the root has indegree $0$, while each of the other nodes has indegree $1$. Thus, we can tweak Equation~(\ref{Eq:degreej}) by substituting $d$ by $d + 1$, and then letting $j = 1$. Under such setting, we find that the probability mass function of $D_{n, 1}$ can be substantially simplified to the following neat and closed form. 
\begin{prop}
	\label{Thm:degree1}
	The probability mass function of the root of a PORT is
	\begin{equation}
	\label{Eq:degree1}
	\Prob(D_{n, 1} = d) = \frac{d(2n - d - 3)!}{2^{n - d - 1} (n - d - 1)! (2n - 3)!!},
	\end{equation}
	for $d = 1, 2, \ldots, n - 1$.
\end{prop}
\begin{proof}
	Recall Equation~(\ref{Eq:degreej}), and set $j = 1$. Replacing $d$ with $d + 1$ in the equation, we have
	$$\Prob(D_{n, 1} = d) = \frac{\Gamma(d + 1)\Gamma\left(\frac{1}{2}\right)\sum_{i = 0}^{d} \frac{(-1)^i \Gamma\left(n - 1 - \frac{i}{2}\right)}{\Gamma(i + 1)\Gamma(d + 1 - i)\Gamma\left(j - 1 - \frac{i}{2}\right)}}{\Gamma\left(n - \frac{1}{2}\right)}.$$
	Reimplementing the strategy of writing the total sum into partial sums with respect to odd indicies and even indicies, we apply the {\em Euler's reflection formula} to gamma functions, and obtain
	\begingroup
	\allowdisplaybreaks
	\begin{align*}
	\\\Prob(D_{n, 1} = d)&= \frac{\Gamma(d + 1)\Gamma\left(\frac{1}{2}\right)}{\Gamma\left(n - \frac{1}{2}\right)}\left(\sum_{{i \tiny{\mbox{ is even}}} \atop {0 \le i \le d}}\frac{\Gamma\left(n - 1 - \frac{i}{2}\right)}{\Gamma(i + 1)\Gamma(d + 1 - i)\Gamma\left(j - 1 - \frac{i}{2}\right)}\right.
	\\&\qquad{} - \left. \sum_{{i \tiny{\mbox{ is odd}}} \atop {0 \le i \le d}}\frac{\Gamma\left(n - 1 - \frac{i}{2}\right)}{\Gamma(i + 1)\Gamma(d + 1 - i)\Gamma\left(j - 1 - \frac{i}{2}\right)}\right)
	\\&= \frac{\Gamma(d + 1)\Gamma\left(\frac{1}{2}\right)}{\Gamma\left(n - \frac{1}{2}\right)} \frac{2^{d + 1} \Gamma(d + 1 - n)}{4^n \Gamma(d) \Gamma(d + 3 - 2n) \cos(n\pi)}
	\\ &= \frac{d (2n - d - 3) \Gamma(2n - d - 3) 2^n}{2^{2n - d - 1} (n - d - 1) \Gamma(n  - d - 1) (2n - 3)!!}
	\\ &= \frac{d(2n - d - 3)!}{2^{n - d - 1} (n - d - 1)! (2n - 3)!!}.
	\end{align*}
	\endgroup
\end{proof}

The probability mass function of $D_{n, 1}$ in Proposition~\ref{Thm:degree1} agrees with that derived in~\cite{ZhangANALCO}. The proof in~\cite{ZhangANALCO} requires massive algebraic computations and simplifications, so the proof given in this paper appears much more concise and succinct. 

\subsection{Moments}
\label{Sec:moments}
In general, the probability mass function (c.f.\ Equation~(\ref{Eq:degreej})) is unwieldy for moment computations. Alternatively, we appeal to a two-color {\em \Polya\ urn model}~\cite{Mahmoud2009} to calculate the mean and variance of $D_{j, n}$. Imagine that there is an urn containing balls of two colors (white and blue). Let $W_n$ be the degree of the node labeled with $j$ (white balls) at time $n \ge j$, and $B_n$ be the total degree of all the other nodes (blue balls). At time $(n + 1)$, if the node labeled with $j$ is selected, $W_n$ increases by one, and $B_n$ also increases by one, which is contributed by the edge incident to the node labeled with $(n + 1)$; if any other node is selected, $B_n$ increases by two. This dynamic can be represented by the following {\em replacement matrix}
\begin{equation}
\label{Mat:replacement}
\begin{pmatrix}
1 & 1
\\	0 & 2
\end{pmatrix}.
\end{equation}
This \Polya\ urn scheme appropriately interprets the mechanism of preferential attachment, as, upon the insertion at time point $n + 1$, the probability of the node labeled with $j$ being selected is exactly $W_n/(W_n + B_n)$. Another equivalent approach to modeling the dynamics of degree change is to employ an {\em extended} PORT. The basic idea is to fill all the gaps in the original tree with {\em external} nodes, which represent insertion positions. We omit the details in this section, but will revisit this strategy in the sequel. 

The replacement matrix (c.f.\ Matrix~(\ref{Mat:replacement})) is triangular, so the \Polya\ urn associated with this kind of replacement matrix is called {\em triangular \Polya\ urn}. Triangular urns are well studied, and the moments of white balls are explicitly characterized in~\cite[Theorem 3.1]{Zhang2015}. We exploit those results to get the following proposition.
\begin{prop}
	\label{Prop:momdegreej}
	For a fixed $1 \le j \le n$ and $n \ge 2$, we have
	\begin{align*}
	\E[D_{n, j}] &= \frac{\Gamma(n)\Gamma\left(j - \frac{1}{2}\right)}{\Gamma\left(n - \frac{1}{2}\right)\Gamma(j)} - \delta_{j, 1},
	\\ \Var[D_{n, j}] &= -\frac{\Gamma^2(n)\Gamma^2\left(j - \frac{1}{2} \right)}{\Gamma^2\left(n - \frac{1}{2}\right) \Gamma^2(j)} - \frac{\Gamma(n)\Gamma\left(j - \frac{1}{2}\right)}{\Gamma\left(n - \frac{1}{2}\right)\Gamma(j)} + \frac{4n - 2}{2j - 1}.
	\end{align*}
\end{prop}

We discover that when $n$ is large, both $\E[D_{n, j}]$ and $\Var[D_{n, j}]$ experience phase transitions. To compute the asymptotic expectation and variance, we apply the {\em Stirling's approximation} to the expectation and variance of $D_{n, j}$ in Proposition~\ref{Prop:momdegreej}. As $n \to \infty$, we have
\begin{align}
\E[D_{n, j}] &\sim \frac{\Gamma \left(j - \frac{1}{2}\right)}{\Gamma(j)} \, n^{1/2}, \label{Eq:limitexpj}
\\ \Var[D_{n, j}] &\sim \left(\frac{4}{2j - 1} - \frac{\Gamma^2\left(j - \frac{1}{2}\right)}{\Gamma^2(j)} \right) \, n - \frac{\Gamma\left(j - \frac{1}{2}\right)}{\Gamma(j)} \, n^{1/2}.\label{Eq:limitvarj}
\end{align}
We keep the second highest order term (i.e., the term that involves $n^{1/2}$) in the asymptotic variance of $D_{n, j}$ because it makes a contribution when $j$ grows in the linear phase (with respect to $n$). We reapply the Stirling's approximation to Equations~(\ref{Eq:limitexpj}) and~(\ref{Eq:limitvarj}), respectively, and obtain the next corollary.

\begin{corollary}
	\label{Cor:phase}
	As $n \to \infty$, we have
	$$\E[D_{n, j}] \sim
	\begin{cases}
	\bigl(\Gamma(j - 1/2)/\Gamma(j)\bigr) \, n^{1/2},
	\qquad &\mbox{for fixed }j,
	\\ (n / j)^{1/2}, \qquad &\mbox{for }j \to \infty,
	\end{cases}$$
	and 
	$$\Var[D_{n, j}] \sim
	\begin{cases}
	\left(\frac{4}{2j - 1} - \frac{\Gamma^2\left(j - 1/2 \right)}{\Gamma^2(j)} \right) \, n,
	\qquad &\mbox{for fixed }j,
	\\ n/j, \qquad &\mbox{for }j \to \infty, j = o(n),
	\\ 1/\theta - 1/\sqrt{\theta}, \qquad  &\mbox{for }j/n = \theta, 0 < \theta < 1.
	\end{cases}$$			
\end{corollary}

The formulation of $\E[D_{n, j}]$ coincides with that developed in~\cite{Szymanski}. In addition, $\Var[D_{n, j}]$ is also reported in~\cite{Szymanski}, where it is presented in terms of a sum of binomial coefficients. In this paper, we provide an alternative approach to determining $\E[D_{n, j}]$ and $\Var[D_{n, j}]$, and both of them are in neat and closed forms.

\section{Zagreb index}
\label{Sec:Zagreb}
A topological index of a graph quantifies it by turning its structure into a number. Capturing structures in numbers allows researchers to compare graphs according to certain criteria. There are many possible indices that can be constructed for static and random graphs. Each index tends to capture certain features of the graphs, such as sparseness, regularity and centrality. Examples of indices that have been introduced for random graphs include the Zagreb index~\cite{Feng2011},
the Rand\'{i}c index~\cite{Feng2010}, the Wiener index~\cite{Fuchs, Neininger}, the Gini index~\cite{Balaji, Zhang}, and a topological index measuring graph weight~\cite{Zhang2016}.

In this section, we investigate the Zagreb index for the class of PORTs. The Zagreb index was first introduced by~\cite{Gutman} in 1972. It has been a popular topological index to study molecules and complexity of selected classes of molecules in mathematical chemistry, and to model quantitative structure-property relationship (QSPR) and quantitative structure-activity relationship (QSAR) in chemoinformatics~\cite{Todeschini}.  

\subsection{Mean and variance}
\label{Sec:rec}

The {\em Zagreb index} of a graph is defined as the sum of the squared degrees of all the nodes therein. Given a PORT at time $n$, $T_n$, its Zagreb index is thus given by
$$Z_n = {\mathbf{Zagreb}}(T_n) = \sum_{j = 1}^{n} D^2_{n, j},$$
where $D_{n, j}$, again, is the degree of the node labeled with $j$ in $T_n$. Let $\indicator(n, j)$ indicate the event that the node labeled with $j$ is selected at time $n$. In the next proposition, we present the exact expectation of $Z_n$ as well as a weak law.

\begin{prop}
	\label{Prop:Zagrebmean}
	The mean of the Zagreb index of a PORT at time $n \ge 1$ is
	$$\E[Z_n] = 2(n - 1) \bigl(\Psi(n) + \gamma \bigr),$$
	where $\Psi(\cdot)$ is the {\em digamma function}, and $\gamma$ is the {\em Euler's constant}. As $n \to \infty$, we have
	$$\frac{Z_n}{n \log{n}} \convP 2.$$ 
\end{prop}

\begin{proof}
	Upon the insertion taking place at time point $n$, we have the following recurrence of $Z_n$ conditional on $\field_{n - 1}$ and $\indicator(n, j)$:
	\begin{equation}
	\label{Eq:Zagreb}
	Z_n = Z_{n - 1} + (D_{n - 1, j} + 1)^2 - D^2_{n - 1, j} + 1,
	\end{equation}
	where the terms $\bigl((D_{n - 1, j} + 1)^2 - D^2_{n - 1, j}\bigr)$ altogether indicate the contribution by the node labeled with $j$ (to the Zagreb index) by the degree increase, and the last term $1$ comes from the contribution by the newcomer (the node labeled with $n$). We simplify Equation~(\ref{Eq:Zagreb}) and take the expectation with respect to $\indicator(n, j)$ to get
	\begin{align*}
	\E[Z_n \given \field_{n - 1}] &= Z_{n - 1} + 2 \sum_{j = 1}^{n - 1} D_{n - 1, j} \times \Prob\bigl(\indicator(n, j)\bigr) + 2
	\\ &=  Z_{n - 1} + 2 \sum_{j = 1}^{n - 1} D_{n - 1, j} \times \frac{D_{n - 1, j}}{2(n - 2)} + 2
	\\ &=  Z_{n - 1} + \frac{\sum_{j = 1}^{n - 1} D^2_{n - 1, j}}{n - 2} + 2
	\\ &= \left(1 +\frac{1}{n - 2}\right) Z_{n - 1} + 2.
	\end{align*}
	Taking another expectation with respect to $\field_{n - 1}$, we receive a recurrence on the mean of $Z_n$, namely,
	$$\E[Z_n] = \frac{n - 1}{n - 2} \, \E[Z_{n - 1}] + 2.$$
	This recurrence is well defined for $n \ge 3$, so we can set the initial condition at $\E[Z_2] = Z_2 = 2$. Solving the recurrence, we obtain the result stated in the proposition. Notice that the result is well defined for all $n \ge 1$, albeit the developed recurrence is undefined for $n = 2$.
	
	As $n \to \infty$, we have $\Psi(n) \sim \log{n}$ asymptotically. Hence, we obtain the following convergence in $L_1$-space:
	$$\frac{Z_n}{n \log{n}} \convL 2.$$
	This convergence takes place in probability as well.
\end{proof}

Towards the computation of the second moment of $Z_n$, we consider a new topological index that is the sum of cubic degrees of nodes in a graph. Let $Y_n = \sum_{j = 1}^{n} D^3_{n, j}$ be such index of $T_n$. In the next lemma, we derive the mean of $Y_n$, and a weak law as well.

\begin{lemma}
	\label{Lem:Y}
	The mean of $Y_n$ of a PORT at time $n \ge 2$ is
	$$\E[Y_n] = \frac{32 \Gamma(n + 1/2)}{\sqrt{\pi}\Gamma(n - 1)} - 6(n - 1)\left(\Psi(n) + \gamma + \frac{8}{3}\right).$$
	As $n \to \infty$, we have
	$$\frac{Y_n}{n^{3/2}} \convP \frac{32}{\sqrt{\pi}}.$$
\end{lemma}

\begin{proof}
	We consider a recurrence for $Y_n$ conditional on $\field_{n - 1}$ and $\indicator(n, j)$, mimicking that for $Z_n$ in Equation~(\ref{Eq:Zagreb}) as follows:
	\begin{align*}
	Y_n &= Y_{n - 1} + (D_{n - 1, j} + 1)^3 - D^3_{n - 1, j} + 1
	\\ &= Y_{n - 1} + 3 D^2_{n - 1, j} + 3 D_{n - 1, j} + 2
	\end{align*}
	Taking the expectation with respect to $\indicator(n, j)$, we get
	\begingroup
	\allowdisplaybreaks
	\begin{align*}
	\E[Y_n \given \field_{n - 1}] &= Y_{n - 1} + 3\sum_{j = 1}^{n - 1} D^2_{n - 1, j} \times \frac{D_{n - 1, j}}{2(n - 2)} + 3\sum_{j = 1}^{n - 1} D_{n - 1, j} \times \frac{D_{n - 1, j}}{2(n - 2)} + 2
	\\ &= Y_{n - 1} + \frac{3}{2(n - 2)} Y_{n - 1} + \frac{3}{2(n - 2)} Z_{n - 1} + 2.
	\end{align*}
	\endgroup	
	Take another expectation with respect to $\field_{n - 1}$ and plug in the result of $\E[Z_{n - 1}]$to receive a recurrence on $\E[Y_n]$:
	$$
	\E[Y_n] = \frac{2n - 1}{2(n - 2)} \, \E[Y_{n - 1}] +  2 \bigl(\Psi(n - 1) + \gamma \bigr) + 4.
	$$
	Solving the above recurrence with initial condition $\E[Y_2]  = Y_2 = 2,$ we obtain the stated result.
	
	Towards the asymptotic of $Y_n$, we apply the Stirling's approximation to $\E[Y_n]$ to get
	$$\E[Y_n] = \frac{32}{\sqrt{\pi}} \, n^{3/2} + O(n \log{n}).$$
	Thus, we obtain an $L_1$ convergence for $Y_n/n^{3/2}$ as well as a weak law.
\end{proof}

Note that in Lemma~\ref{Lem:Y}, the expression of $\E[Y_n]$ is well defined for $n \ge 2$. As $n \to 1$, $\Gamma(n - 1)$ in the denominator of the first term approaches infinity, and $(n - 1)$ in the second term approaches $0$, rendering $\E[Y_n] \to 0$. This is consistent with the fact of $\E[Y_1] = Y_1 = 0$, as there is an isolated node in the tree. 

We are now ready to calculate the second moment of $Z_n$ as well as the variance of $Z_n$.
\begin{prop}
	\label{Prop:Zagrebvar}
	The second moment of the Zagreb index of a PORT at time $n \ge 1$ is
	$$\E \left[Z^2_n\right] = 4(n \log{n})^2 + 8 \gamma \left(n^2 \log{n}\right) + \left(16 + 4 \gamma^2 - \frac{2 \pi^2}{3}\right) n^2 + O\left(n^{3/2}\right),$$
	and the variance of $Z_n$ is
	$$\Var[Z_n] =  \left(16 - \frac{2 \pi^2}{3}\right) n^2 + O \left(n^{3/2}\right).$$
\end{prop}

\begin{proof}
	We revisit the almost-sure recurrence for $Z_n$ in Equation~(\ref{Eq:Zagreb}) and square both sides to get
	$$
	Z^2_n = Z^2_{n - 1} + 4 D^2_{n - 1, j} + 4 + 4 Z_{n - 1} D_{n - 1, j} + 4 Z_{n - 1} + 8 D_{n - 1, j}.
	$$
	Averaging it out with respect to $\indicator(n, j)$, we have
	\begin{align*}
	\E\left[Z^2_{n} \given \field_{n - 1}\right] &= Z_{n - 1}^2 + 4 \sum_{j = 1}^{n} D^2_{n - 1, j} \times \frac{D_{n - 1, j}}{2(n - 2)} + 4 
	\\ &\qquad{}+ 4 Z_{n - 1}\sum_{j = 1}^{n} D_{n - 1, j} \times \frac{D_{n - 1, j}}{2(n - 2)} + 4Z_{n - 1}
	\\ &\qquad\qquad{}+ 8\sum_{j = 1}^{n} D_{n - 1, j} \times \frac{D_{n - 1, j}}{2(n - 2)}
	\\ &= Z_{n - 1}^2 + \frac{2}{n - 2} Y_{n - 1} + 4 + \frac{2}{n - 2} Z^2_{n - 1} + 4 Z_{n - 1} + \frac{4}{n - 2} Z_{n - 1}
	\\ &= \frac{n}{n - 2} Z^2_{n - 1} + \frac{2}{n - 2} Y_{n - 1} + \frac{4(n - 1)}{n - 2} Z_{n - 1} + 4.
	\end{align*}
	The recurrence for $\E\left[Z^2_n\right]$ is thus obtained by taking the expectation of the formula above with respect to $\field_{n - 1}$ both sides, and by plugging in the results of $\E[Y_n]$ and $\E[Z_n]$. Solving the recurrence with initial condition $\E\left[Z^2_2\right] = Z^2_2 = 4$, we get the solution of $\E\left[Z^2_2\right]$ stated in the proposition. In what follows, we obtain the variance of $Z_n$ by computing $\E\left[Z^2_2\right] - \E^2[Z_n]$.
\end{proof}
Notice that $Z_n^2$ converges to $4(n \log{n})^2$ in $L_1$-space as well as in probability, both directly from the {\em Continuous Mapping Theorem}. Besides,we would like to point out that we derive the exact solution of $\E\left[Z^2_2\right]$, but do not present it in the manuscript for better readability. However, the exact solution is available upon request.

\subsection{Investigation of asymptotic behavior}
\label{Sec:normal}
A sharp concentration on the variance of a random variable usually suggests asymptotic normality. In this section, we investigate the asymptotic behavior of $Z_n$. We provide several persuasive evidence to show that $Z_n$ does not converge to normal for large $n$, significantly different from the Zagreb index of random recursive trees~\cite{Feng2011}.

Notice that $Z_n$'s are not independent random variables. In general, a common approach to assessing (plausible) Gaussian law for a sequence of dependent random variables is to consider a transformation on the random variables such that the new sequence is a martingale array, and then to apply the {\em Martingale Central Limit Theorem} on martingale differences~[Theorem 3.2, Corollary 3.1]\cite{Hall} or its extensions.
 
\begin{lemma}
	\label{Lem:martingale}
	For $n \ge 1$, the sequence consisting of
	$$M_n = \frac{2}{n - 1} Z_n -4\bigl(\Psi(n) + \gamma\bigr)$$
	is a martingale.
\end{lemma}

\begin{proof}
	We consider two constant sequences $\{\alpha_n\}_n$ and $\{\beta_n\}_n$ such that the following martingale property holds for all $n \ge 3$.
	\begin{align*}
	\E[\alpha_n Z_n + \beta_n \given \field_{n - 1}] &= \alpha_{n} \E[Z_n \given \field_{n - 1}] + \beta_n
	\\ &= \frac{\alpha_n (n - 1)}{n - 2} Z_{n - 1} + 2 \alpha_n + \beta_n.
	\\ &= \alpha_{n - 1} Z_{n - 1} + \beta_{n - 1}
	\end{align*}
		This produces two recurrences on $\alpha_n$ and $\beta_n$, respectively,
	$$\alpha_n = \frac{n - 2}{n - 1} \, \alpha_{n - 1} \qquad \mbox{and} \qquad \beta_n = \beta_{n - 1} - 2 \alpha_n,$$
	with arbitrary choices of initial conditions. We thus obtain the following solutions
	$$\alpha_n = \frac{2}{n - 1} \qquad \mbox{and} \qquad \beta_n = -4\bigl(\Psi(n) + \gamma\bigr),$$
	by choosing initial conditions $\alpha_3 = 1$ for the former and $\beta_1 = 0$ for the latter, respectively. However, notice that the result stated in the lemma is well defined for all $n \ge 1$.
\end{proof} 

Noting that the martingale $M_n$ is equivalent to
$$M_n = \frac{Z_n - \E[Z_n]}{(n - 1)/2},$$
we have
$$\E[M_n] = 0 \qquad \mbox{and} \qquad \E\left[M_n^2\right] = \frac{\Var[Z_n]}{(n - 1)^2/4} = 64 - \frac{8 \pi^2}{3} < + \infty,$$
leading to the fact that $\{M_n\}_n$ is a mean-zero and square-integrable martingale. Moreover, according to the {\em Doob's Martingale Convergence Theorem}, there is an $L_2$-measurable random variable, to which $M_n$ converges; So is $Z_n$ after properly scaled. Let us define martingale differences, which are expressed in terms of a difference operator, $\nabla M_j = M_j - M_{j - 1}$. In the next lemma, we show that $|\nabla M_j|$'s are uniformly bounded for all $j$.

\begin{lemma}
	The terms $|\nabla M_j|$ are absolutely uniformly bounded for $j = 2, 3, \ldots n$.
\end{lemma}

\begin{proof}
	By the construction of the martingale in Lemma~\ref{Lem:martingale}, we have
	\begin{align*}
	|\nabla M_j| &= |M_j - M_{j - 1}|
	\\ &= \bigl|\alpha_j Z_j + \beta_j - (\alpha_{j - 1} Z_{j - 1} + \beta_{j - 1})\bigr| 
	\\ &\le \alpha_{j - 1} \left| \frac{\alpha_j}{\alpha_{j - 1}} Z_j - Z_{j - 1}\right| + |\beta_j - \beta_{j - 1}|
	\\ &= \alpha_{j - 1} \left| \frac{j - 2}{j - 1} Z_j - Z_{j - 1}\right| + \frac{4}{j - 1}
	\\ &\le \alpha_{j - 1} | Z_j - Z_{j - 1}| + \frac{\alpha_{j - 1}}{j - 1} Z_j + \frac{4}{j - 1}
	\\ &\le \frac{2(2j - 3)}{j - 2} + \frac{2 \bigl( (j - 1)^2 + (j - 1) \bigr)}{(j - 1)(j - 2)} + \frac{4}{j - 1}
	\\ &= \frac{6j^2 - 8j - 2}{(j - 1)(j - 2)},
	\end{align*}
	which is strictly decreasing for $j \ge 3$.
\end{proof}

To use the Martingale Central Limit Theorem\cite[Corollary 3.1]{Hall}, two conditions need to be verified. The first is known as the {\em Lindeberg's condition}, given by
$$U_n := \sum_{j = 1}^{n} \E \left[\left(\nabla M_j\right)^2 \indicator\bigl(\left|\nabla M_j\right| > \varepsilon|\bigr) \given \field_{j - 1}\right] \convP 0.$$
The Lindberg's condition is not satisfied, albeit we can show $|\nabla M_j| \convP 0$. What is needed indeed is a stronger statement: $\max_j|\nabla M_j| \convP 0$, which is apparently not true in our case.

Besides, we look at the conditional variance condition, given by
	$$V_n := \sum_{j = 1}^{n} \E \left[\left(\nabla M_j\right)^2 \given \field_{j - 1}\right] \convP \eta^2,$$
where $\eta^2$ is a finite almost-sure random variable. We compute $V_n$ exactly, and get
$$V_n = \left[\left(64 - \frac{8 \pi^2}{3}\right)n + O \left(n^{1/2}\right)\right],$$
which converges only when scaled by $n$. 

As neither the Lindeberg's condition nor the conditional variance condition is verified, it is evident that $M_n$ is not normally distributed as $n \to \infty$. 

In addition, we conduct a numeric study to further investigate the asymptotic behavior of $Z_n$. We generate $5000$ independent PORTs at time $50000$, compute the Zagreb index of each of the simulated trees, and plot the estimated density function in Figure~\ref{Fig:sim} via a kernel method~\cite{Silverman}.
\begin{figure}
	\begin{center}
		\includegraphics[scale=0.37]{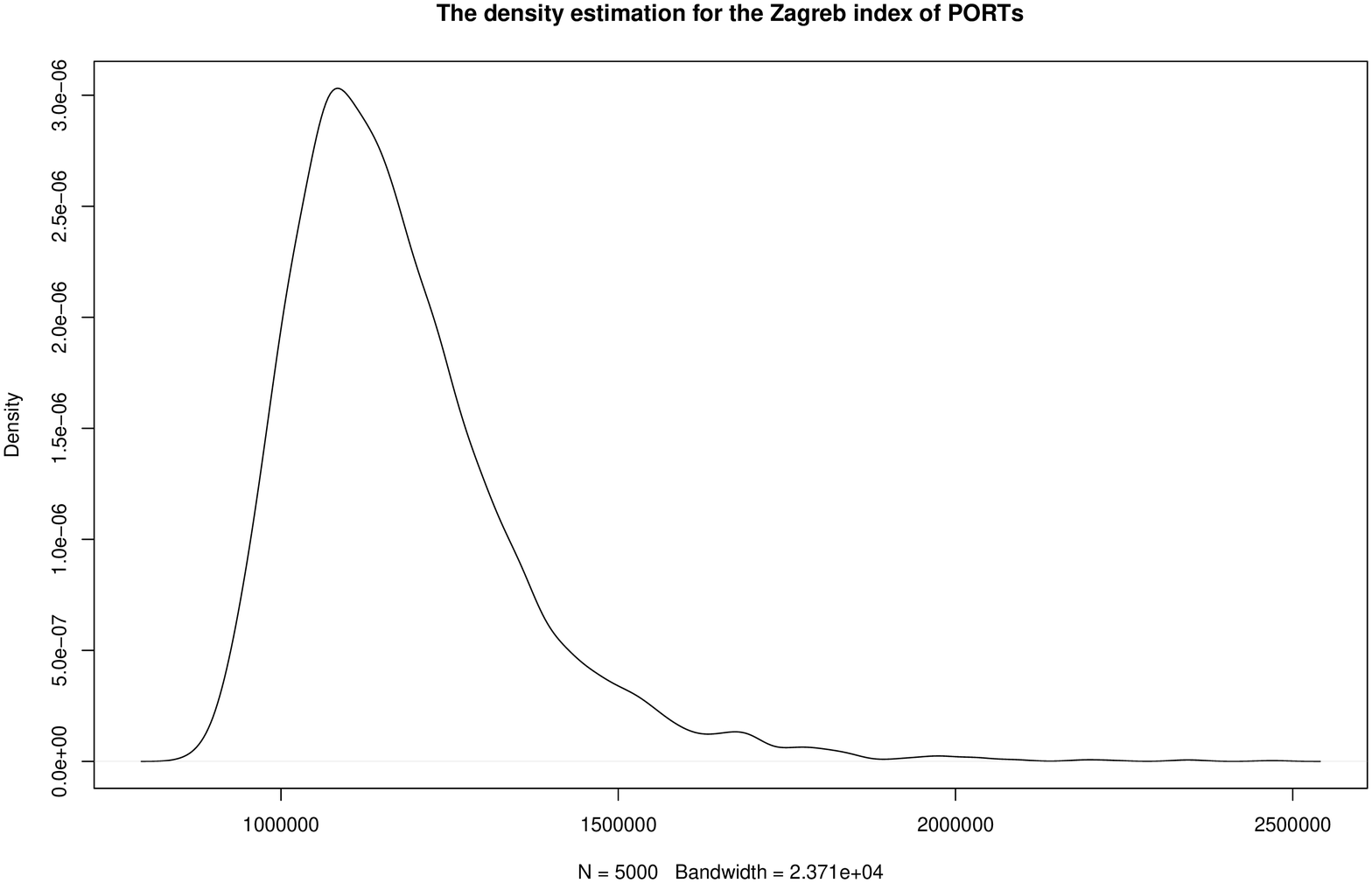}
	\end{center}
\caption{}
\label{Fig:sim}
\end{figure}
The estimated density is obviously skewed to the left, suggesting that $Z_n$ is not normally distributed for large $n$. Lastly, we run the {\em Shapiro-Wilk} test on dataset consisting of the $5000$ Zagreb indices of the simulated PORTs. The $p$-value is less than $2.2 \times 10^{-16}$, strongly against the null hypothesis for assuming normality.

In~\cite{Feng2011}, the authors proved that the limiting distribution of the Zagreb index of random recursive trees is normal with mean $6n$ and variance $8n$. Based off our investigation, we conjecture that asymptotic normality does not exist for PORTs. A possible reason that causes skewness may be the feature of the growth of the model, i.e., preferential attachment. A node that already has large degree in the tree is more and more likely to attract subsequent newcomers, potentially destroying the balance of growth as well as symmetry.

\section{Poissonized plane-oriented recursive trees}
\label{Sec:Pois}
Many real structures do not grow in discrete time, but in continuous time. In this section, we study PORTs embedded into continuous time. The embedding is done by changing the {\em interarrival} times between node additions from equispaced discrete units to more general renewal intervals. One~\cite{Athreya} suggests to use exponential random variables as interarrival time. Under this choice, a count of the arrival points constitutes a {\em Poisson process}~\cite{Ross}. Hence, such embedding is commonly called {\em Poissonization}~\cite{Aldous}. The advantage of Poissonization is that the underlying exponential random variables (interarrival times) share an appealing property---the {\em memoryless} property.

We elaborate the growth of a Poissonized PORT by employing an extended graph analogous to that for random recursive trees in~\cite{Mahmoud1992}. As mentioned in Section~\ref{Sec:moments}, an extended PORT is obtained by filling all the gaps with external nodes. Under the Poissonization framework, each external node is endowed with an independent clock that rings in $\Expo(1)$. When the clock of an external node rings, a newcomer joins in the tree, and is connected with the node (in the original tree) that carries that external node by an edge. Then, all the new gaps are filled by new external nodes instantaneously. Upon each renewal, the clocks of existing external nodes are reset owing to the memoryless property, and the new external nodes come endowed with their own independent clocks. We do not consider the time loss of the execution of node additions. Thus, this growth process is {\em Markovian}.

To investigate the degree distribution of the node labeled with $j$, we assume that $t_0$, the time of its first appearance in the tree, is finite. At this point, there is $1$ external node carried by the node labeled with $j$, and we paint it white; Meanwhile, there is a total of $(2j - 3)$ external nodes carried by all the other nodes, and we paint them blue. In the two-color \Polya\ urn framework, the dynamic of ball addition (at each renewal point) is analogous to that for the discrete-time counterpart, so it also can be represented by Matrix~(\ref{Mat:replacement}). The feature of preferential attachment is reflected in the number of external nodes adjacent to the nodes from the original tree. Let $W(t)$ and $B(t)$ be the numbers of white and blue balls (external nodes) at time $t \ge t_0$, respectively. Noting that $W(t)$ is exactly equal to the degree of the node labeled with $j$ at time $t$, we thus place our focus on the distribution of $W(t)$.

Recall that Matrix~(\ref{Mat:replacement}) is triangular, the associated (Poissonized) \Polya\ urn process is also called triangular \Polya\ process. This class of urn models was recently investigated by~\cite{ChenChen}. In this source, the moment generating function of $W(t)$ under a more general framework was developed. For our specific setting, we present the moment generating function of $W(t)$ in the next proposition.

\begin{prop}
	At time $t \ge t_0$, the moment generating function of $W(t)$ is
	$$\phi_{W(t)}(u) = \frac{e^{u - (t - t_0)}}{1 - (1 - e^{-(t - t_0)}) e^u}.$$
\end{prop}

This result is obtained directly from~\cite[Lemma 4.3]{ChenChen} by plugging in appropriate parameters. Then, the $r$th moment of $W(t)$ can be derived from $\phi_{W(t)}(u)$ for all $r \ge 1$. The expression of the $r$th moment of $W(t)$ is available but not in a closed form, rather in a partial sum of an alternating sequence involving {Stirling numbers of the second kind} and gamma functions. Thus, we do not present all the moments of $W(t)$ in this paper, but only the first two moments (after simplification) and accordingly the variance in the next corollary. 

\begin{corollary}
	At time $t$, the first moment, second moment and variance of $W(t)$ respectively are
	\begin{align*}
	\E[W(t)] &= e^{t - t_0},
	\\ \E\left[W^2(t)\right] &= 2e^{2(t - t_0)} - e^{t - t_0},
	\\ \Var[W(t)] &= e^{2(t - t_0)} - e^{t - t_0}.
	\end{align*}
\end{corollary}

Noticing that the probability distribution of a random variable is uniquely determined by its moment generating function provided it exists, we determine the asymptotic distribution of $W(t)$ after properly scaled.

\begin{theorem}
	As $t \to \infty$, we have
	$$W(t) / e^{t} \convD \Expo\left(1/e^{t_0}\right).$$
\end{theorem}

\begin{proof}
	According to the moment generating function of $W(t)$, we derive the moment generating function of $\tilde{W}(t) = W(t)/e^{t - t_0}$ as follows:
	$$
	\phi_{\tilde{W}(t)}(u) = \E\left[e^{(u/e^t) W(t)}\right] = \frac{e^{u/e^{t - t_0} - (t - t_0)}}{1 - (1 - e^{-(t - t_0)}) e^{u/e^{t - t_0}}},
	$$
	which converges to $1/(1 - u)$ as $t \to \infty$. Noticing that $1/(1 - u)$ is the moment generating function of $\Expo(1)$, we thus have
	$$W(t) / e^{t - t_0} \convD \Expo(1).$$
	The result stated in the theorem follows the scaling property of exponential random variables.
\end{proof}

\section{Concluding remarks}
\label{Sec:concl}
In the last section, we add some concluding remarks and propose some future work. In this paper, we investigate three properties of PORTs. First, we determine the degree distribution of a node with a fixed degree by developing its probability mass function. Additionally, we compute the first two moments by exploiting a two-color triangular urn. 

Second, we look into the Zagreb index of the class of PORTs. We calculate the exact mean and variance via recurrence methods. Weak laws are developed as well. Towards the asymptotic behavior of the Zagreb index, we conjecture that it is not normal, significantly different from that for random recursive trees. We substantiate our conjecture by showing the invalidity of the two sufficient conditions for the Central Limit Theorem as well as a numeric example. One of our future work is to develop a rigorous method in support of our conjecture. Besides, we plan to investigate many other topological indices, such as the Gini index and the Randi\'{c} index, for PORTs in the future research.

Last, we study the degree profile of PORTs embedded into continuous time, so-called Poissonized PORTs. We interpret the growth of Poissonized PORTs by introducing extended trees. The exact moment generating function of the degree variable is determined. We show that the asymptotic distribution of the degree variable scaled by $e^t$ is exponential.

%\bibliographystyle{amsplain}
%\bibliography{yourbibfilename}

% add below the content of your .bbl file produced by bibtex.

\end{document}